\documentclass[11pt,a4paper]{article}
\usepackage{amsthm}
\usepackage{amsmath,amsfonts,amssymb,graphics,verbatim,fullpage,bbm}
\usepackage[usenames,dvipsnames]{color}

\ifx\pdfoutput\undefined
\usepackage[dvips,bookmarks]{hyperref}
\else
\usepackage{hyperref}
\fi
\hypersetup{%
pdftitle={A Levy input fluid queue with input and workload regulation},
pdfauthor={Palmowski, Vlasiou, Zwart},
pdfkeywords={stochastic systems},
bookmarks=true,%
bookmarksnumbered=true,
pdfstartview={FitH},
colorlinks=true,
citecolor=Plum,
linkcolor=RedOrange,
urlcolor=RawSienna,
bookmarksopen=true,%
bookmarksopenlevel=1,
unicode=true,%
breaklinks=true,%
}%

\def\te#1{\mathrm{e}^{#1}}
\def\td{\text{\rm d}}
\numberwithin{equation}{section}

\newtheorem{theorem}{Theorem}[section]
\newtheorem{lem}{Lemma}[section]

\newcommand{\eqdistr}{\stackrel{{\cal D}}{=}}

\theoremstyle{definition}
\newtheorem*{remark}{Remark}
\newtheorem{Ex}{Example}

\begin{document}
\title{A L\'{e}vy input fluid queue with input and workload regulation}
\author{Zbigniew Palmowski\footnote{University of Wroc\l aw, pl.\ Grunwaldzki 2/4, 50-384 Wroc\l aw, Poland, E-mail: zpalma@math.uni.wroc.pl} \qquad Maria Vlasiou\footnote{\textsc{Eurandom} and Department of Mathematics and Computer Science; Eindhoven University of Technology; P.O.\ Box 513; 5600 MB Eindhoven; The Netherlands, E-mail: m.vlasiou@tue.nl}\qquad Bert Zwart \footnote{Centrum Wiskunde \& Informatica, Science Park 123, Amsterdam, The Netherlands, Email: bert.zwart@cwi.nl}}
\maketitle

\begin{abstract}
We consider a queuing model with the workload evolving between consecutive i.i.d.\ exponential timers $\{e_q^{(i)}\}_{i=1,2,\ldots}$ according to a spectrally positive L\'{e}vy process $Y_i(t)$ that is reflected at zero, and where the environment $i$ equals 0 or 1. When the exponential clock $e_q^{(i)}$ ends, the workload, as well as the L\'evy input process, are modified; this modification may depend on the current value of the workload, the maximum and the minimum workload observed during the previous cycle, and the environment $i$ of the L\'evy input process itself during the previous cycle. We analyse the steady-state workload distribution for this model. The main theme of the analysis is the systematic application of non-trivial functionals, derived within the framework of fluctuation theory of L\'evy processes, to workload and queuing models.
\end{abstract}

\section{Introduction}

In this paper we focus on a particular queuing system with state-dependent two-stage regulation mechanism. There has been considerable previous work on queues with state-dependent service and arrival processes; see for example the survey by Dshalalow~\cite{dshalalow97} for several references. The model under consideration involves a reflected  L\'evy process connected to the evolution of the workload. Special cases of L\'evy processes are the compound Poisson process, the Brownian motion, linear drift processes, and independent sums of the above.

Specifically, in this paper we consider a storage/workload model in which the workload evolves according to a spectrally positive L\'{e}vy process $Y_0(t)$ or $Y_1(t)$, both reflected at zero. That is, let $X_i(t)$ ($i=0, 1$) be two independent spectrally positive L\'{e}vy processes (i.e.\ L\'{e}vy processes with only positive jumps) modelling the input minus the output of the process and define $-\inf_{s\leqslant 0} X_i(s)=0$ and $X_i(0)=x\geqslant 0$. Then we have that $Y_i(t)=X_i(t)-\inf_{s\leqslant t} X_i(s)$ (where $Y_i(0)=x$ for some initial workload $x\geqslant 0$). In addition, at exponential times with intensity $q$, given by $\{e_q^{(i)}\}_{i=1,2,\ldots}$, the workload is ``reset'' to a certain level, depending on the workload level just before the exponential clock ends, its minimum and maximum in the previous cycle, and the environment of the reflected L\'{e}vy process. Specifically, at epoch $t=e^{(1)}_q+\cdots+e^{(n)}_q$, the workload $V(t)$ equals $F_n^{(1)}(V(t-))$ for some random nonnegative i.i.d.\  functionals $F_n: [0,\infty)\times \{0,1\}^3 \rightarrow [0,\infty)\times \{0,1\} $, where $F^{(1)}_n$ and $F^{(2)}_n$ denote the first and second coordinate of $F_n$, respectively. Moreover, at these moments the feedback information is available
and the label of the Laplace exponent is changed according to functional $F^{(2)}$.

This model unifies and extends several related models in various directions. First, in a previous paper \cite{palmowski11}, we considered the case where $X_0=X_1$, and where the workload correction depended on the current workload only. Second, if $X_0=X_1=X$ is a compound Poisson process and if $F_n$ is the identity function, then our model reduces to the workload process of the M/G/1 queue. Next, Kella et al.~\cite{kella03} considered a  model with workload removal, which fits into our model by taking $F_n(x)=0$ and by letting the  spectrally positive L\'{e}vy process $X_0=X_1=X$ be a Brownian motion superposed with an independent compound Poisson component. Our work is also related to papers \cite{bekker08, bekker09}, who considered the adaptation of the Laplace exponent of the input process at the moments when the workload process crosses level $K$. Bekker et al.~\cite{bekker09} analysed also the adaptation of the process at Poisson instants but without the possibility of additional regulation at these moments and without taking into account the supremum of the workload process between exponential timers.

The model we consider can be also seen as a L\'evy-driven queue. Queues and L\'evy processes have received significant attention recently, as described in the excellent survey by D\c{e}bicki and Mandjes \cite{debicki12}. Indeed, a point of this paper is that one can use the joint law of a reflected L\'evy process at an exponential time as well as it supremum and infimum to effectively analyse queuing and inventory models where specific workload corrections take place based on whether an extremely high and/or low workload has been reached since the previous point of inspection. We discuss several examples to illustrate this point. In addition, we derive two types of qualitative results. We prove stability by showing that an embedded Markov chain is Harris recurrent. The proof involves analysing the two-step transition kernel. In addition, we consider the tail behavior of the invariant distribution, building on results in \cite{palmowski11}. The examples for which we analyse the entire distribution are, respectively, a clearing-type model, a model related to the TCP protocol, and an inventory model.

The paper is organised as follows. In Section~\ref{s:Model description} we demonstrate model we deal with. In Section~\ref{s:Preliminaries} we introduce a few basic facts concerning spectrally positive L\'{e}vy processes. In Section~\ref{s:embedded} we consider the embedded workload process and derive a recursive equation for its stationary distribution. By the PASTA property this determines the steady-state workload distribution. Later on, in Section~\ref{s:examples} we present some special cases. Finally, in Section \ref{tail} we focus on the tail behaviour of the steady-state workload.


\section{Model description}\label{s:Model description}

We consider a L\'evy input model, with the additional feature that at exponential $(q)$ times both the input process and the workload can be corrected.

 Let $e_q^{(i)}$ be an i.i.d.\ sequence of exponential $(q)$ random variables. Let $S_n = e^{(1)}_q+\cdots +e^{(n)}_q$, and let $N(t) = \max\{n: S_n\leqslant t\}$ be a Poisson process with rate $q$. At times $S_n$, $n\geqslant 1$, inspection takes place by a controller, who can decide to change the workload input for the interval $[S_n,S_{n+1})$, as well as the workload level itself. In any given interval $[S_n,S_{n+1})$, the workload process $V(t)$ is driven by a spectrally positive L\'evy process with Laplace exponent $\psi_0$ or $\psi_1$, depending on the choice of the controller at time $S_n$.

Apart from the workload level $V(t)$, the choice of the controller can depend on the following quantities: let $K$ and $L$ be two constants, called the upper and lower threshold and define
\begin{align*}
J_S(t) &= \mathbbm{1}( \sup_{r\in [S_{N(t)},t)} V(r) \geqslant K ),\\
J_I(t) &= \mathbbm{1}( \inf_{r\in [S_{N(t)},t)} V(r) \leqslant L ),\\
J_E(t) &\in \{0,1\},\\
W(t) &=(V(t),J_S(t),J_I(t),J_E(t)).
\end{align*}
During $[S_n,S_{n+1})$, $V(t)$ is driven by a spectrally positive L\'evy process with exponent $\psi_{J_E(t)}$.
At time $S_{n+1}-$, the workload level and the input process are modified according to a random function $F_{n+1}: [0,\infty)\times \{0,1\}^3\rightarrow [0,\infty)\times \{0,1\} $, where $F^{(1)}_{n+1}$ and $F^{(2)}_{n+1}$ denote the first and second coordinate of $F_{n+1}$, respectively. Specifically, the workload at time $S_{n+1}$ becomes
\begin{equation*}
V(S_{n+1}) = F^{(1)}_{n+1}(W(S_{n+1}-)),
\end{equation*}
and the Laplace exponent of the L\'evy input process $J_E(S_{n+1})$ becomes
\begin{equation*}
J_E(S_{n+1}) = F^{(2)}_{n+1}(W(S_{n+1}-)).
\end{equation*}
The idea of the analysis is as follows. We are interested in the steady state distribution of the process we have just defined. By PASTA, it suffices to analyse the embedded Markov chain $W_n= W(S_{n}-)=:(U_n, {J}_n)$ with ${J}_n=(J_{S}(S_{n}-), J_{I}(S_{n}-), J_{E}(S_{n}-))$ ($n\geqslant 1$). Thus, $U_n=V(S_n-)$. The transition kernel of this Markov chain can be computed by building on various recent results on reflected spectrally one-sided L\'evy processes. The next section will focus on this.

The remainder of this section mentions some illustrative examples of the function $F\stackrel{d}{=}F_1$. Let $u\geqslant 0$ and $\bar j=(j_1,j_2,j_3)$.  Examples on how one may change the workload are:
\begin{enumerate}
\item $F^{(1)}(u, \bar j)= B_{\bar j}$ (this includes clearing models, where $B_{\bar j}=0$, considered in Kella et al.~\cite{kella03});
\item $F^{(1)}(u, \bar j)= u+B_{\bar j}$ (inventory models where additional goods are ordered, for example $R$ units are ordered if the inventory has dropped below a specific value $L$, leading to $B_{\bar j} =R\cdot \mathbbm{1}(j_2=1)$);
\item $F^{(1)}(u, \bar j)= (B_{\bar j}-u)^+$ (considered in Vlasiou \& Palmowski \cite{palmowski11});
\item $F^{(1)}(u, \bar j)= u\cdot \mathbbm{1}(j_2=0)+ R\cdot \mathbbm{1}(j_2=1)$ (set inventory to level $R$ if it drops below $L$ and do nothing otherwise).
\end{enumerate}
Examples one could think of for the second coordinate are of the type that input is slowed down when a high level is reached, and/or speed up when a low level is reached. Suppose that $Y_1(t)$ represents a high input with respect to $Y_0(t)$. Then, examples of how to modify the second coordinate are:
\begin{enumerate}
\item $F^{(2)}(u,j_1,1,j_3)=1$, $F^{(2)}(u,j_1,0,j_3)=0$ (choose high input if the workload was below $L$ and choose low input otherwise);
\item $F^{(2)}(u,1,j_2,j_3)=0$, $F^{(2)}(u,0,j_2,j_3)=1$ (choose high input if the workload did not exceed $K$ and choose low input otherwise);
\item $F^{(2)}(u,j_1,j_2,j_3)=\mathbbm{1}(u<K)$ (choose high input iff the workload was less than $K$);
\item $F^{(2)}(u,j_1,j_2,j_3)=j_3$ (never change environment).
\end{enumerate}

In Section~\ref{s:examples}, we investigate a subset of these combinations. Additionally, we consider the case $F^{(1)}(u, \bar j)=\delta u$ (TCP). Further examples can be found in the literature, in particular in \cite{bekker08,bekker09}. Note that in some of the models considered in those papers, the workload is adapted instantaneously when a high level is reached, rather than after an exponential time. In principle one could recover such results from our model by letting $q\rightarrow\infty$.

\section{Preliminaries on L\'evy processes}\label{s:Preliminaries}

Let $X_i(\cdot)$, $i=0,1$, be two spectrally positive L\'evy processes. Throughout this paper we exclude the case where the L\'evy processes $X_i$ have monotone paths. Let the dual process of $X_i(t)$ be given by $\hat{X}_i(t)=-X_i(t)$. The process $\{\hat{X}_i(s), s\leqslant t\}$ is a spectrally negative L\'{e}vy process and has the same law as the time-reversed process $\{X_i((t-s)-)-X_i(t), s\leqslant t\}$. Following standard conventions, let $\underline{X}_i(t)=\inf_{s\leqslant t}X_i(s)$, $\overline{X}_i(t)=\sup_{s\leqslant t}X_i(s)$ and similarly $\underline{\hat{X}}_i(t)=\inf_{s\leqslant t}\hat{X}_i(s)$, and $\overline{\hat{X}}_i(t)=\sup_{s\leqslant t}\hat{X}_i(s)$. One can readily see that the processes $Y_i(t)=X_i(t)-\underline{X}_i(t)$ (for $Y_i(0)=0$) and  $\overline{X}_i(t)$ (where $X_i(0)=0$) have the same distribution; see e.g.\ Kyprianou~\cite[Lemma 3.5, p.\ 74]{kyprianou-ILFLP}. Moreover,
$$
-\underline{X}_i(t)\eqdistr \overline{\hat{X}}_i(t),\qquad \overline{X}_i(t)\eqdistr -\underline{\hat{X}}_i(t).
$$

Since the jumps of $\hat{X}_i$ are all non-positive, the moment generating function $E[\te{\theta \hat{X}_i(t)}]$ exists for all $\theta \geqslant 0$ and is given by $E[\te{\theta \hat{X}_i(t)}] = \te{t\psi_i(\theta)}$. The Laplace exponent $\psi_i(\theta)$ is well defined at least on the positive half-axis where it is strictly convex with the property that $\lim_{\theta\to\infty}\psi_i(\theta)=+\infty$. Moreover, $\psi_i$ is strictly increasing on $[\Phi_i(0),\infty)$, where $\Phi_i(0)$ is the largest root of $\psi_i(\theta)=0$. We shall denote the right-inverse function of $\psi_i$ by $\Phi_i:[0,\infty)\to[\Phi(0),\infty)$.

Denote by $\sigma_i$ the Gaussian coefficient and by $\nu_i$ the L\'{e}vy measure of $\hat{X}_i$ (note that $\sigma_i$ is also a Gaussian coefficient of $X_i$ and that $\Pi_i(A)=\nu_i(-A)$ is a jump measure of $X_i$).  Throughout this paper we assume that the following (regularity) condition is satisfied:
\begin{equation}\label{eq:condW}
\sigma_i > 0 \quad \text{or}\quad \int_{-1}^0 x \nu_i(\td x) = \infty \quad \text{or}\quad \nu_i(\td x)<<\td x
\end{equation}
for $i=0,1$, where $<< \td x $ means absolutely continuity with respect to the Lebesgue measure.
Finally, $P_{i,x}$ denotes the probability measure $P$ under the condition that $J_E(0)=i$ for the label process $J_E$, $X_i(0)=x$, and $E_{i,x}$ indicates the expectation with respect to $P_{i,x}$. If $X_i(0)=0$ we will skip the subscript $x$.

\subsection{Scale functions}\label{ss:scale}
For $q\geqslant0$ and $i=0,1$, there exists a function $W^{(q)}_i: [0,\infty) \to [0,\infty)$, called the {\it $q$-scale function}, that is continuous and increasing with Laplace transform
\begin{equation}
\int_0^\infty \te{-\theta y} W^{(q)}_i (y)  \td y = (\psi_i(\theta) - q)^{-1},\qquad\theta > \Phi_i(q).
\end{equation}
The domain of $W^{(q)}_i$ is extended to the entire real axis by setting $W^{(q)}_i(y)=0$ for $y<0$. We mention here some properties of the function $W^{(q)}_i$
that have been obtained in the literature which we will need later on.

On $(0,\infty)$ the function $y\mapsto W^{(q)}_i(y)$ is right- and left-differentiable and, as shown in \cite{lambert00}, under  condition \eqref{eq:condW}, it holds that $y\mapsto W^{(q)}_i(y)$ is continuously differentiable for $y>0$.

Closely related to $W^{(q)}_i$ is the function $Z^{(q)}_i$ given by
$$
Z^{(q)}_i(y) = 1 + q\int_0^yW^{(q)}_i(z)\td z.
$$
The name ``$q$-scale function'' for $W^{(q)}_i$ and $Z^{(q)}_i$ is justified as these functions are harmonic for the process $\hat{X}$ killed upon entering $(-\infty,0)$. Here, we give a few examples of scale functions. For a large number of examples of scale functions see e.g.\ Chaumont et al.~\cite{chaumont09}, Hubalek and Kyprianou~\cite{hubalek11}, Kyprianou and Rivero~\cite{kyprianou08a}.

\begin{Ex}\label{Brownianwithdrift}
If $X(t)=X_1(t)=X_2(t) = \sigma B(t) - \mu t$ is a Brownian motion with drift $\mu$ (a standard model for small service requirements) then
$$
W^{(q)}(x) =W^{(q)}_1(x)=W^{(q)}_2(x)= \frac{1}{\sigma^2\delta} [\te{(-\omega+\delta) x} -
\te{-(\omega+\delta)x}],
$$
where $\delta = \sigma^{-2}\sqrt{\mu^2 + 2q\sigma^2}$ and $\omega = \mu/\sigma^2$.
\end{Ex}

\begin{Ex}\label{eg:compound}
Suppose
\begin{equation}\label{compound}
X(t)=X_1(t)=X_2(t)= \sum_{i=1}^{N(t)}\sigma_i-pt,
\end{equation}
where $p$ is the speed of the server and $\{\sigma_i\}$ are i.i.d.\  service times that are coming according to the Poisson process $N(t)$ with intensity $\lambda$. We assume that all $\sigma_i$ are exponentially distributed with mean $1/\mu$.  Then $\psi(\theta)=p\theta-\lambda \theta/(\mu+\theta)$ and the scale function of the dual $W^{(q)}=W^{(q)}_1=W^{(q)}_2$ is given by
\begin{equation}\label{scale:compound}
W^{(q)}(x) = p^{-1}\left(A_+ \te{q^+(q)x} - A_- \te{q^-(q)x}\right),
\end{equation}
where $A_\pm = \frac{\mu + q^\pm(q)}{q^+(q)-q^-(q)}$ with
$q^+(q)=\Phi(q)$ and $q^-(q)$ is the smallest root of $\psi(\theta)=q$:
$$
q^\pm(q)=\frac{{q} + \lambda -\mu p \pm \sqrt{({q} + \lambda-\mu p)^2 + 4 p {q}\mu}}{2 p}.
$$
\end{Ex}

\subsection{Fluctuation identities}\label{ss:identities}

The functions $W^{(q)}_i$ and $Z^{(q)}_i$ ($i=0,1$) play a key role in the fluctuation theory of reflected processes as shown by the following identity (see Bertoin~\cite[Th.\ VII.4 on p.\ 191 and (3) on p.\ 192]{bertoin-LP} or Kyprianou and Palmowski~\cite[Th.\ 5]{kyprianou05}).

\begin{lem}\label{exitident}
For $\alpha >0$,
$$
E\left( \te{-\alpha \overline{X}_i(e_{q})}\right) =\frac{q(\alpha -\Phi_i \left( q\right) )}{\Phi_i \left( q\right) (\psi_i \left( \alpha \right) -q)}\ ,
$$
which is equivalent to
\begin{equation*}
P(\overline{X}_i(e_q)\in \td x)=\frac{q}{\Phi_i \left( q\right) } W^{(q)}_i(\td x)-qW^{(q)}_i(x)\td x,\qquad x>0.
\end{equation*}
Moreover, $-\underline{X}_i(e_q)$ follows an exponential distribution with parameter $\Phi(q)$.
\end{lem}

Let $\tau_i^0=\inf\{t\geqslant 0: X_i(t)<0\}$. The $q$-scale function gives also the density $r^{(q)}_i(x,y)=R^{(q)}_i(x,\td y)/\td y$ of the $q$-potential measure
\begin{equation}\label{Rqgeneral}
R^{(q)}_i(x,\td y):=\int_0^\infty \te{-qt}P_{i,x}(X_i(t)\in \td y, \tau^0_i>t)\td t
\end{equation}
of the process $X_i$ killed on exiting $[0,\infty)$ when initiated from $x$; see also Pistorius~\cite{pistorius} and Palmowski and Vlasiou \cite{palmowski11}.

\begin{lem}\label{potentialkilled}
Under \eqref{eq:condW}, we have that
$$
r^{(q)}_i(x,y)= \int_{[(x-y)^+,x]}\te{-\Phi_i(q) z}\left[W^{(q)\prime}_i(y-x+z)-\Phi_i(q)W^{(q)}_i(y-x+z)\right]\td z.
$$
\end{lem}

\begin{remark}
Lemma~\ref{potentialkilled} and similar results can be proven without the assumption made in \eqref{eq:condW}, but at the cost of more complex expressions. We would have to use \eqref{Rqgeneral} instead of the much nicer form for $r_i^{(q)}(x,y)\td y$.
\end{remark}

Lemma~\ref{potentialkilled} implies the following useful formula describing the density of the reflected L\'evy process at an exponential time, which is taken from \cite[Lemma 3.1]{palmowski11}.

\begin{lem}\label{transitionreflected}
We have that
$P_{i,x}(Y_i(e_q)\in \td y)=h_i(x,y)\td y+\te{-\Phi(q)x}W_i^{(q)}(0)\delta_0(\td y)$, where
\begin{equation*}
h_i(x,y)=qr_i^{(q)}(x,y)+ \te{-\Phi_i(q)x}\left[\frac{q}{\Phi_i
\left(q\right) } W_i^{(q)\prime}(y)-qW_i^{(q)}(y)\right],
\end{equation*}
and where  $r_i^{(q)}(x,y)$ is given in
Lemma~\ref{potentialkilled}.
\end{lem}

We need an extension of this result, proved by Pistorius~\cite[Th.\ 1(ii), p.\ 99]{pistorius}.

\begin{lem}\label{jointdensitysup}
Under \eqref{eq:condW}, we have that
$$ P_{i,x}(Y_i(e_q)\in \td y, \sup_{s\leqslant e_q}Y_i(s)\leqslant K)={q}\left\{h_{i,K}^{(q)}(x,0)\delta_0(\td y)+h_{i,K}^{(q)}(x,y)\td y\right\}$$
for
$$
h^{(q)}_{i,K}(x,y)=W^{(q)}_i(K-x)\frac{W^{(q)\prime}_i(y)}{W^{(q)\prime}_i(K)}-W^{(q)}_i(y-x).
$$
\end{lem}

We also need the following expression for the joint law of the reflected process and its infimum, which follows directly from
Lemma \ref{potentialkilled} (after shifting the trajectory downwards by $L$).

\begin{lem}\label{jointdensityinf}
Under \eqref{eq:condW}, we have that, for $x,y\geqslant L$,
$$ P_{i,x}(Y_i(e_q)\in \td y, \inf_{s\leqslant e_q}Y_i(s)> L)=
qr^{(q)}_i(x-L,y-L)\td y.
$$
\end{lem}

We are now ready to provide the formulae that will be  crucial in our analysis in the next sections.

\begin{lem}\label{jointdensity-kernels}
Under \eqref{eq:condW}, we have that:
\begin{align}
\kappa_{0,0,i}(x,\td y) :&=P_{i,x}(Y_i(e_q)\in \td y, \sup_{s\leqslant e_q}Y_i(s)\leqslant K, \inf_{s\leqslant e_q}Y_i(s)> L )\nonumber\\ &=
q\left[\frac{W^{(q)}_i(K-x)W^{(q)}_i(y-L)}{W^{(q)}_i(K-L)} - W^{(q)}_i(y-x)\right] \td y, \label{kernelfirst}\\
\kappa_{0,1,i}(x,\td y) :&=P_{i,x}(Y_i(e_q)\in \td y, \sup_{s\leqslant e_q}Y_i(s)\leqslant K, \inf_{s\leqslant e_q}Y_i(s)\leqslant L )\nonumber\\ &=
q\left[h^{(q)}_{i,K}(x,0)\delta_0(\td y) + h^{(q)}_{i,K}(x,y)\td y\right]- \kappa_{0,0,i}(x,\td y),\nonumber\\
\kappa_{1,0,i}(x,\td y) :&=P_{i,x}(Y_i(e_q)\in \td y, \sup_{s\leqslant e_q}Y_i(s)> K, \inf_{s\leqslant e_q}Y_i(s)> L )\nonumber\\ &=
qr^{(q)}_i(x-L,y-L)\td y
- \kappa_{0,0,i}(x,\td y),\nonumber\\
\kappa_{1,1,i}(x,\td y) :&=P_{i,x}(Y_i(e_q)\in \td y, \sup_{s\leqslant e_q}Y_i(s)> K, \inf_{s\leqslant e_q}Y_i(s)\leqslant L )\nonumber\\
 &= h_i(x,y)\td y+\te{-\Phi(q)x}W_i^{(q)}(0)\delta_0(\td y) - \sum_{\stackrel{k,l} {k\cdot l=0}}\kappa_{k,l,i}(x,\td y).\nonumber
\end{align}
\end{lem}

\begin{proof}
Equality \eqref{kernelfirst} follows from \cite[Eq.\ (13)]{pistorius} and \cite{suprun76}. The other three identities follow from the law of total probability and application of the previous lemmas.
\end{proof}

\section{Equilibrium distribution of the embedded process}\label{s:embedded}

Recall that  $U_n=V(S_n-)$ and that ${J}_n=(J_{S}(S_{n}-), J_{I}(S_{n}-), J_{E}(S_{n}-))$ and define $(U,J)$ as the weak limit of $(U_n,J_n)$, assuming it exists and that it is unique. The main goal is to derive an expression for the distribution of $(U,J)$. To this end, we define the transform $v_{\bar j}(s)= E[\te{-sU} I(J=\bar j)]$. We now derive an equation for
this transform:
\begin{align}
v_{\bar l}(s) &= \sum_{\bar j \in \{0,1\}^3} \int_{x=0}^\infty
E_x\left[\te{-sY_{l_3}(e_q)}, J_S(e_q-)=l_1,J_I(e_q-)=l_2 \right]  P( F(U,J)\in \td x,l_3 ; J=\bar j)\nonumber\\
&= \sum_{\bar j \in \{0,1\}^3} \int_{x=0}^\infty \int_{y=0}^\infty \int_{u=0}^\infty \te{-sy} \kappa_{\bar l}(x,y) P( F(u,\bar j)\in (\td x,l_3)) P(U\in \td u; J=\bar j).\label{eq:stationary}
\end{align}

In the next section, we  consider several examples for which it is possible to solve this set of equations. Of course, it helps if the $q$-scale function $W^{(q)}$ is explicitly known, and in many cases (for example in the case of a Brownian motion superposed with a compound Poisson process with phase-type jumps), the $q$-scale function can be written as a sum of exponentials, allowing a tractable analysis (as illustrated by the first example in the next section).

Before we start looking for solutions, we first determine whether a solution actually exists, i.e.\ we discuss stability conditions. Let $m_i$ be the drift of the L\'evy process $i$. Recall that we assume that our L\'{e}vy processes are assumed not to be subordinators. The following assumptions are not the most general possible, but seem to cover most practical purposes, in particular the examples discussed later on. Assume there exists a constant $M>0$ and random pairs $(A,B)$ such that for $u\geqslant M$ and $\overline{j}\in \{0,1\}^3$:
\begin{equation}
F^{(1)}(u,\overline{j}) \stackrel{\mathcal{D}}{\leqslant} Au+B =:F^a(u).
\end{equation}
Also assume that $F^{(2)}(u,\overline{j})= 0$ for $u\geqslant M$ and finally, assume that there exists a random variable $D_M$ such that
\[
F^{(1)}(u,\overline{j}) \stackrel{\mathcal{D}}{\leqslant} D_M, \qquad u\leqslant M, \quad i=1,2, \quad \overline{j}\in \{0,1\}^3.
\]

\begin{theorem}
Assume that neither $X_i(\cdot)$ nor $-X_i(\cdot), i=0,1$ are subordinators. Then the distribution of $(U_n,J_n)$ converges in total variation if one of the following conditions is satisfied:
\begin{enumerate}
\item $E[\log A]<0$, $E[\log D_M]<\infty$ and $E[\log B]<\infty$;
\item $A=1$, $E[D_M]<\infty$ and $E[B]+m_0/q<0$.
\end{enumerate}
\end{theorem}

\begin{proof}
We shall show that the process $W_n=(U_n,J_n)$ is a Harris chain. We focus on the first set of assumptions (the second set is similar, but easier).

Set $\epsilon = -E[\log A]/2$.
Let $M'>\max\{K,M\}$ be a constant with the following properties: $E[\log (A+(B+\overline{X}_1(e_q))/x)]<-\epsilon$  if $x>M'$.
Also, $M'$ is chosen large enough such that 
$P(X_i(e_q) \in (M',2M'))>0$ for $i\in\{0,1\}$.
Finally, $M'$ is chosen large enough such that $P(D_M<M')>0$.

Define the set $R=[0,M'] \times \{0,1\}\times \{0,1\} \times \{0,1\}$.
The function $f(u,\overline{j})=\log u$ is a Lyapounov function for the Markov chain $(U_n,J_n)$. For $u>M'$ we have
\begin{align*}
E[f(U_1,J_1) \mid U_0=u; J_0=\overline{j}] - f(u, \overline{j})&\leqslant E[\log (Au+B+\overline{X}_{1}(e_q))]-\log u \\
    & = E[\log (A+(B+\overline{X}_1(e_q))/u)]<-\epsilon.
\end{align*}
If we let $D_M$ be independent of $(A,B)$ and define $D'_M = D_M+AM'+B$, we see that
\[
F^{(1)}(u,\overline{j}) \stackrel{\mathcal{D}}{\leqslant} D_M', \qquad u\geqslant M', \quad \overline{j}\in \{0,1\}^3.
\]
In addition, $E[\log D'_M]<\infty$.
The above two considerations imply that the expected return time to $R$ is finite so that $R$ is recurrent.
We construct a nontrivial measure $Q$ and constant $p$ such that
\begin{equation*}
P((U_2,J_2) \in F \mid (U_0,J_0) = (u,\overline{j}) ) \geqslant p Q(F), \hspace{1cm} (u,\overline{j}) \in R,
\end{equation*}
which together with the fact that $R$ is recurrent implies Harris ergodicity, cf.\ \cite[Th.\ VII 3.6]{asmussen-APQ}. We define $Q$ as follows, for $F=F_0\times F_1\times F_2\times F_3 \subseteq [0,\infty) \times \{0,1\}^3$, we set
\[
Q(F) = P_0(Y_0(e_q) \in F_0, 1 \in \cap_{i=1}^3 F_i \mid \sup_{s<e_q} Y_0(s)\geqslant K ).
\]
To construct $p$, we need to create an event that guarantees the lower bound. This will be the intersection of several consecutive events. First we make sure we are above level $M'$ at the end of the first period so that the environment $J_{2,E}$ is going to be equal to $0$. The probability of this to happen is not smaller than $\min_iP(X_i(e_q)\in (M',2M'))$. Second, given that $U_0 < M$, given the increase of $X_i(e_q)$, and given our assumptions on the boundedness of $F^{(1)}$ we arrive at $P(U_1 \in (M',3M')) \geqslant P(D'_M < M')\min_iP(X_i(e_q)\in (M',2M)) =: p_1$. If $U_1>M'$, then $J_{2,E}=1$ as required. Define now
\[
p_2 = P(A3M'+B<4M')P_{4M'}( \inf_{s<e_q} Y_0(s) = 0) P_0( \sup_{s<e_q} Y_0(s) > K),
\]
i.e.\ we make sure that the workload process both hits $0$ (in particular, downcrosses level $L$) and upcrosses larger than $K$. Observe that $p_2>0$ in view of the fact that neither the L\'evy processes nor their duals are subordinators. Defining $p=p_1p_2$ then completes the proof.
 The proof in the second case is similar, taking the Lyapounov function $f(u,\overline{j})=u$.
\end{proof}

This theorem does not yield optimal conditions, but it suffices for a large number of examples, and implies not only existence, but also uniqueness of the stationary distribution as well as convergence in total variation. Optimal conditions generally can be found for specific examples using standard techniques from the literature as surveyed for example in \cite{foss04}.

\section{Examples}\label{s:examples}

We now turn to analysing a few specific examples.

\subsection{A generalised clearing model}

We start with the following simple example. We compute the case where $F^{(1)}$ is given by the first example we have listed in Section~\ref{s:Model description} and $F^{(2)}$ is given by the second example. Namely, we consider the following functional:
$$
F(u,\bar{j})=(B_{\bar{j}}, 1-j_1),
$$
where we further assume that $B_{\bar{j}}$ is exponentially distributed with rate $\mu_{\bar{j}}$ and that the L\'evy input processes are compound Poisson process with exponential jumps. In other words, when the supremum of the workload process during one cycle crossed level $K$, in the next cycle the input process will switch to the lower input process; otherwise we keep the regular input process in the next cycle. Notice that although we do not modify the input process based on the lower threshold $L$, we allow the correction $F^{(1)}$ to depend on the whole environment $(J_S(t),J_I(t),J_E(t))$, and thus also on the lower threshold. E.g., one may choose an exponential correction with a higher mean if the lower threshold has been crossed. The input processes are of the type \eqref{compound}, and thus we have from \eqref{scale:compound} that scale functions of their duals are of the form
$$
W^{(q)}_i(x) = A_{i,1}\te{q_{i,1} x} + A_{i,2}\te{q_{i,2}x}
$$
for each process $Y_i$, $i=0,1$, and where the appropriate constants $A_{i,j}$ and $q_{i,j}$ can be easily derived through minor modifications of the expressions in Example~\ref{eg:compound}.

From \eqref{scale:compound} and \eqref{eq:stationary} we have:
\begin{align}\label{eq:Maria1}
\nonumber
v_{\bar l}(s) &= \sum_{\stackrel{j_1=1-l_3}{j_2,j_3}} \int_{x=0}^\infty P(B_{\bar j}\in \td x) E_x\left[\te{-sY_{l_3}(e_q)}, J_S(e_q-)=l_1,J_I(e_q-)=l_2 \right]\\
&= \sum_{\stackrel{j_1=1-l_3}{j_2,j_3}} \int_{x=0}^\infty \int_{y=0}^\infty \te{-sy} \kappa_{\bar l}(x,\td y) P(B_{\bar j}\in \td x) P({J}=\bar{j}),
\end{align}
where by definition $v_{\bar j}(0)=P({J}=\bar{j})$. These equations are explicit in the sense that everything at the right-hand side is known. Since the corrections $B_{\bar j}$ we consider are exponentially distributed, and the kernels $\kappa_{\bar l}$ ultimately depend on the scale functions, which are again a sum of exponentials, the right-hand side reduces to simply integrating exponential functions. The computations, although lengthy, are straightforward. Below we give some intermediate key steps towards the final result.

Since the corrections may depend on the lower threshold, we need to compute all eight transforms $v_{\bar l}$, which reduces to computing the double integral at the right-hand side of \eqref{eq:Maria1} for all four kernels $\kappa_{\bar l}$ from Lemma~\ref{jointdensity-kernels}, and then computing the resulting equations at $s=0$ in order to determine the unknown probabilities $P({J}=\bar{j})$. This procedure will result to a linear $8\times8$ system which uniquely determines the unknown constants.

For simplicity, define
$$
z_{\bar l}(s)=\int_{x=0}^\infty \int_{y=0}^\infty \te{-sy} \kappa_{\bar l}(x,\td y) P(B_{\bar j}\in \td x).
$$
We then have:
\begin{equation*}
z_{\bar l}(s)=\int_{x=L}^K\int_{y=L}^K \te{-sy} P(B_{\bar j}\in \td x) q\left[\frac{W^{(q)}_{l_3}(K-x)W^{(q)}_{l_3}(y-L)}{W^{(q)_{l_3}}(K-L)} - W^{(q)}_{l_3}(y-x)\right] \td y;
\end{equation*}
see also Lemma~\ref{jointdensity-kernels}. Observe that in this case the two thresholds have not been crossed, which means that no mass is assigned outside $[L,K]$. Straightforward computations lead to
\begin{multline*}
z_{0,0,l_3}(s)=\sum_{i,m=1}^2 A_{l_3,i} A_{l_3,m} \frac{q \te{-q_{l_3,i}L}}{W^{(q)}_{l_3}(K-L)} \frac{\te{-(s-q_{l_3,i})L}-\te{-(s-q_{l_3,i})K}}{s-q_{l_3,i}}\frac{\te{q_{l_3,m}K} \mu_{\bar j}}{\mu_{\bar j}+q_{l_3,m}}\times\\
\left(\te{-(\mu_{\bar j}+q_{l_3,m})L}-\te{-(\mu_{\bar j}+q_{l_3,m})K}\right)
-\sum_{n=1}^2 \frac{q A_{l_3,n}}{s-q_{l_3,n}}\bigg[\frac{\mu_{\bar j}}{\mu_{\bar j}+q_{l_3,n}}\left(\te{-(s+\mu_{\bar j})L}-\te{-(s-q_{l_3,n})L-(\mu_{\bar j}+q_{l_3,n})K}\right)\\
-\frac{\mu_{\bar j}}{\mu_{\bar j}+s}\left(\te{-(s+\mu_{\bar j})L}-\te{-(s+\mu_{\bar j})K}\right)\bigg].
\end{multline*}
Likewise,
\begin{multline*}
z_{0,1,l_3}(s)=\sum_{i=1}^2 \left[A_{l_3,i} \frac{q \mu_{\bar j}}{\mu_{\bar j}+q_{l_3,i}}\frac{W^{(q)\prime}_{l_3}(0)}{W^{(q)\prime}_{l_3}(K)}\left[\te{q_{l_3,i}K}-\te{-\mu_{\bar j}K}\right]\right]\\
+\frac{q}{W^{(q)\prime}_{l_3}(K)}\sum_{i,m=1}^2 A_{l_3,i} A_{l_3,m} \left[\te{q_{l_3,i}K}-\te{-\mu_{\bar j}K}\right]\frac{\mu_{\bar j}}{\mu_{\bar j}+q_{l_3,i}}\frac{q_{l_3,m}}{s-q_{l_3,m}}\left[1-\te{-(s-q_{l_3,m})K}\right]\\
-\frac{q\mu_{\bar j}}{W^{(q)\prime}_{l_3}(K)}\sum_{i,m=1}^2 A_{l_3,i} A_{l_3,m} \left[\frac{\te{q_{l_3,i}K}-\te{-(\mu_{\bar j}+s)K}}{(s-q_{l_3,m})(\mu_{\bar j}+q_{l_3,i}+s)}-\frac{\te{-(s-q_{l_3,i}-q_{l_3,m})K}-\te{-(\mu_{\bar j}+s)K}}{(s-q_{l_3,m})(\mu_{\bar j}+q_{l_3,i}+q_{l_3,m})}\right]-z_{0,0,l_3}(s)
\end{multline*}
and
\begin{multline*}
z_{1,0,l_3}(s)=q \mu_{\bar j}\te{-(\mu_{\bar j}+s)L}\frac{s-\Phi_{l_3}(q)}{(\mu_{\bar j}+\Phi_{l_3}(q))(\mu_{\bar j}+s)}\sum_{i=1}^2 A_{l_3,i}\frac{q_{l_3,i}-\Phi_{l_3}(q)}{(s-q_{l_3,i})(s-\Phi_{l_3}(q))}-z_{0,0,l_3}(s).
\end{multline*}
Finally,
\begin{multline*}
z_{1,1,l_3}(s)=\frac{\mu_{\bar j}}{\mu_{\bar j}+\Phi_{l_3}(q)} \bigg(W^{(q)}_{l_3}(0)+\frac{q}{\Phi_{l_3}(q)}\Big(\frac{s}{\psi_{l_3}(s)-q}-W^{(q)}_{l_3}(0)\Big)-\frac{q}{\psi_{l_3}(s)-q}\bigg)\\
+\mu_{\bar j}\sum_{i=1}^2 A_{l_3,i} \frac{q_{l_3,i}-\Phi_{l_3}(q)}{(\mu_{\bar j}+s)(\mu_{\bar j}+s+\Phi_{l_3}(q)-q_{l_3,i})}\left(\frac{q}{\mu_{\bar j}+\Phi_{l_3}(q)}+\frac{1}{s}\right)-\sum_{\stackrel{m,n\in\{0,1\}}{m\cdot n=0}}z_{m,n,l_3}(s).
\end{multline*}

\subsection{A TCP-like control}

In this subsection we consider the following situation: we take $L=0,K=\infty$ and have
\begin{equation*}
F^{(1)}(u,\bar j) = \delta u.
\end{equation*}
This case is a generalisation of a model for the throughput behaviour of a data connection under the Transmission Control Protocol (TCP) where typically the L\'{e}vy process is a simple deterministic drift; see for example \cite{altman99,guillemin04,maulik06,maulik09} and references therein.

We additionally assume the workload input is changed according to the following rule: the input is set to state 1 if the workload process reaches 0. Thus, $F^{(2)}(u,\bar j) = j_2$, with $L=0$ and $K=\infty$. The special case where the input process is kept constant is analysed in \cite{vlasiou10}.

Our general framework simplifies as we can drop the environment parameter w.r.t.\ $K$ and simply keep track of states of the form $(u,j_I,j_E)$. Thus, the workload level is always shrunk by a factor $\delta$ and the input process is set to 1 if and only if the system emptied, otherwise the input process is set to 0.

Define $v_{ij}(s) = E[\te{-sU}; J_S=i, J_E=j]$. Observe that
\[
v_{0a}(s) = \sum_{j} \int_{x=0}^\infty P(\delta U \in \td x ; J=(a,j))\int_{y=0}^\infty \te{-sy} P_x(X_a(e_q) \in \td y; \tau_0^- > e_q).
\]
Since
\[
P_x(X_a(e_q) \in \td y; \tau_0^- > e_q) = P_x(X_a(e_q) \in \td y) - P_0(X_a(e_q) \in \td y) P_x(\tau_0^- < e_q),
\]
we obtain
\begin{align*}
v_{0a}(s) &= \sum_{j} \int_{x=0}^\infty P(\delta U \in \td x ; J=(a,j))\left[ E[\te{-s(X_a(e_q)+x)}]-E_x[\te{-q\tau_0^-}]E[\te{-sX_a(e_q)}] \right]\\
&= \frac{q}{q-\psi_a(s)} \sum_j \int_{x=0}^\infty P(\delta U\in \td x; J=(a,j)) [\te{-sx} - \te{-\phi_a(q) x}]\\
&= \frac{q}{q-\psi_a(s)} \sum_j  [v_{aj}(\delta s)- v_{aj}(\delta \phi_a(q))].
\end{align*}

Using similar arguments we get
\begin{align*}
v_{1a}(s) &= \sum_{j} \int_{x=0}^\infty P(\delta U \in \td x ; J=(a,j))\int_{y=0}^\infty \te{-sy} P_x(Y_a(e_q) \in \td y, \tau_0^-<e_q)\\
 &= \sum_{j} \int_{x=0}^\infty P(\delta U \in \td x ; J=(a,j))\int_{y=0}^\infty \te{-sy} P_0(Y_a(e_q) \in \td y) P_x( \tau_0^-<e_q)\\
 &= \sum_{j} \int_{x=0}^\infty P(\delta U \in \td x ; J=(a,j)) \frac{q}{\phi_a(q)}\frac{s-\phi_a(q)}{\psi_a(s)-q} \te{-\phi_{a}(q)x}  \\
 &= \frac{q}{\phi_a(q)}\frac{s-\phi_a(q)}{\psi_a(s)-q} \sum_j v_{aj}(\delta \phi_a(q)).
\end{align*}

To solve these equations, we need to consider two issues. First, we need to determine the unknown constants $v_{ij}(\delta \phi_i(q))$. In addition, the equation for $v_{0,0}(s)$ is of the form $f(s)=g(s)f(\delta s) + h(s)$.
Such an equation has the solution $$f(s) = f(0) \prod_{j=0}^\infty g(\delta^j s) + \sum_{k=0}^\infty h(\delta^k s) \prod_{j=0}^{k-1} g(\delta^j s),$$ yielding for our case
\begin{equation}\label{v00}
v_{00}(s) = v_{00}(0) \prod_{j=0}^\infty \frac{q}{q-\psi_0(\delta^j s)} + \sum_{k=0}^\infty (v_{01}(\delta^ks)-v_{01}(0))\prod_{j=0}^k \frac{q}{q-\psi_0(\delta^j s)}.
\end{equation}
The infinite product is well defined since $\psi_0(\delta^js) \rightarrow \psi_0(0)=1$ geometrically fast as $j\rightarrow \infty$. Note further that
\begin{align*}
v_{01}(s)&=\frac{q}{q-\psi_1(s)}  [v_{10}(\delta s)+v_{11}(\delta s) - v_{10}(\delta \phi_1(q))-v_{11}(\delta \phi_1(q))],\\
v_{10}(s) &= \frac{q}{\phi_0(q)}\frac{s-\phi_0(q)}{\psi_0(s)-q} [v_{00}(\delta \phi_0(q))+v_{01}(\delta \phi_0(q))],\\
v_{11}(s) &= \frac{q}{\phi_1(q)}\frac{s-\phi_1(q)}{\psi_1(s)-q} [ v_{10}(\delta \phi_1(q))+v_{11}(\delta \phi_1(q))].
\end{align*}
We now sketch how one can obtain eight equations to solve for the unknown constants $v_{ij}(\delta \phi_i(q))$, and $v_{ij}(0)$, $i,j=0,1$. Three equations can be obtained by setting $s=0$ in the expressions for $v_{01}(s)$, $v_{10}(s)$, $v_{11}(s)$. The fourth equation is $\sum_{i,j}v_{ij}(0)=1$. Next, rewrite \eqref{v00} by plugging in the expression for $v_{10}(\cdot)$. Then, the last four equations can be obtained by taking $s= \phi_i(q)$ in the equation for $v_{ij}(s)$.

\subsection{An inventory model}

Consider a spectrally positive L\'evy process with $\psi'(0)<0$ (negative drift) with the additional feature that extra content is fed into the system as soon as it drops below level $L$. This has to be ordered, and comes with a delay which is exponentially distributed with rate $q$. If such an order is outstanding, it is not possible to make another order. The extra amounts consist of i.i.d.\ random variables, distributed according to a generally distributed random variable $B$ with LST $\beta(s)$, for which we assume $P(B>L)=1$. This model fits into our framework with $K=\infty$, $\psi_0=\psi_1$ (so we drop the subscript), and
$F^{(1)}(x,j_I)=x+j_IB$ (we can drop the indices $j_S$ and $j_E$).

Define $v_j(s) = E[\te{-sU}; J_I=j]$. We obtain the following equations for $v_0(s)$ and $v_1(s)$.
\begin{align*}
v_1(s) &=  \sum_{j=0}^1 \int_{x=0}^\infty \td P(U+Bj\leqslant x J_I=j)E_x\left[\te{-sY(e_q)} I(\tau_L<e_q)\right]\\
v_0(s) &=  \sum_{j=0}^1 \int_{x=0}^\infty \td P(U+Bj\leqslant x J_I=j)E_x\left[\te{-sY(e_q)} I(\tau_L>e_q)\right]
\end{align*}
This equations can be simplified by exploiting the fact that $B\geqslant L$ and $U\geqslant L$ if $J_I=0$: the measure $dP(U+Bj\leqslant x J_I=j)$ assigns no mass to $[0,L]$ in this case. For $x>L$, we can write
\begin{align*}
E_x\left[\te{-sY(e_q)} I(\tau_L<e_q)\right] &= P_x(\tau_L<e_q) E_L[\te{-sY(e_q)}] =\te{-\phi(q)(x-L)}E_L[\te{-sY(e_q)}]\\
E_x\left[\te{-sY(e_q)} I(\tau_L>e_q)\right] &= E_x[\te{-sX(e_q)}]-P_x(\tau_L<e_q)E_L[\te{-sX(e_q)}]\\
&=\te{-sx}\frac{q}{q-\psi(s)} - \te{-sL} \frac{q}{q-\psi(s)} \te{-\phi(q)(x-L)}.
\end{align*}
Substituting these expressions in the equations for $v_j(s)$, we obtain
\begin{align*}
v_1(s) &=  \te{L\phi(q)} E_L[\te{-sY(e_q)}]\left( v_0(\phi(q))+v_1(\phi(q))\beta(\phi(q))\right).   \\
v_0(s) &= \frac{q}{q-\psi(s)}\left[v_0(s)+\beta(s)v_1(s)-\te{-L(s-\phi(q))}\left(v_0(\phi(q))+\beta(\phi(q))v_1(\phi(q))\right)\right].
\end{align*}
This is a system with two equations and two unknowns, apart from the additional unknown constant $w(\phi(q))$, with
\[
w(s) = v_0(s)+v_1(s)\beta(s),
\]
which is the workload LST after correction. This constant can be obtained explicitly. Observe that
\[
w(s)=\beta(s)\te{L\phi(q)} E_L[\te{-sY(e_q)}] w(\phi(q)) + \frac q{q-\psi(s)}[w(s)-\te{-L(s-\phi(q))}w(\phi(q))],
\]
 from which we obtain
\begin{equation*}
w(s)=\frac{w(\phi(q)) \te{L\phi(q)}}{\psi(s)} \left[\psi(s) \beta(s) E_L[\te{-sY(e_q)}]+q(\te{-Ls}- \beta(s) E_L[\te{-sY(e_q)}])\right].
\end{equation*}
Since $w(0)=1$, we can apply l'H\^{o}spital's rule to obtain
\begin{equation*}
w(\phi(q))=\te{-L\phi(q)} \left(1+ \frac{L+E[B]+E_L[Y(e_q)]}{-\psi'(0)}\right)^{-1}.
\end{equation*}

\section{Tail behaviour}\label{tail}

In this section we consider the tail behaviour of $V=V(\infty)$, under a variety of assumptions on the tail behaviour of the L\'{e}vy measure $\nu$. We will follow ideas included in \cite{palmowski11}.

Before we present our main results, we first state some useful preliminary results. Again, we exploit that, by PASTA, $V$ has the same law as $U$.

\begin{lem}\label{tailbounds1}
The following (in)equalities hold (in distribution):
\begin{align}
U &\eqdistr \sum_{i=0,1}\max \{F^{(1)}(U, \overline{J}) + X_i(e_q), \overline X_i (e^{(i)}_q)\}\mathbbm{1}(F^{(2)}(U, \overline{J})=i),\label{taillemma1}\\
P(U>x)&= \sum_{i=0,1}{P}(F^{(2)}(U, \overline{J})=i)\bigg(P(X_i(e_q)+F^{(1)}(U, \overline{J})>x) \nonumber\\
      &+\int_0^\infty (P(\overline X_i(e_q)>x)-P(\overline X_i(e_q)>x+y)) P(-\underline{X}_i(e_q)-F^{(1)}(U, \overline{J})\in \td y)\bigg),\label{bound1}\\
P(U>x)&\leqslant \sum_{i=0,1}{P}(F^{(2)}(U, \overline{J})=i)\nonumber\\&\qquad\left(P(X_i(e_q)+F^{(1)}(U, \overline{J})>x)+P(\overline X_i(e_q)>x)P(-\underline{X}_i(e_q)>F^{(1)}(U, \overline{J}))\right),\label{bound3}\\
P(U>x)&\geqslant \sum_{i=0,1}\left(P(\overline X_i(e_q)>x)P(-\underline{X}_i(e_q)>F^{(1)}(U, \overline{J}))\right){P}(F^{(2)}_i(U, \overline{J})=i).\label{bound4}
\end{align}
\end{lem}

\begin{proof}
If $U_n=z$ and $(J_{S,n}, J_{I,n}, J_{E,n})=\overline{j}$, we see that $U_{n+1} \eqdistr Y_i(e_q)$, with $Y_i(0)=F^{(1)}_n(z, \overline{j})$ and $i=F^{(2)}_n(z, \overline{j})$ ($i=0,1$).
Further, we have:
\[
Y_i(t) \eqdistr \max \{Y_i(0)+ X_i(t), \overline X_i(t)\};
\]
for details see \cite[Lemma 3.2]{palmowski11}. This identities produce Equation \eqref{taillemma1}. The other inequalities follow straightforwardly from \eqref{taillemma1}.
\end{proof}

We now turn to the tail behaviour of $U$. For $i=0,1$, let $\Pi_i(A)=\nu_i(-A)$ be the L\'{e}vy measure of the spectrally positive L\'{e}vy process $X_i$ (with support on $R_+$). We first investigate the case where the   L\'{e}vy measure is a member of the convolution equivalent class $\mathcal{S}^{(\alpha)}$ defined below. For $\alpha \geqslant 0$, we say that the measure $\Pi$ is {\it convolution equivalent} ($\Pi \in \mathcal{S}^{(\alpha)}$) if for fixed $y$ we have that
\begin{align*}
\lim_{u\to\infty}\frac{\bar \Pi(u-y)}{\bar \Pi(u)}&=\te{\alpha y},       &&\mbox{if $\Pi$ is nonlattice,}\\
\lim_{n\to\infty}\frac{\bar \Pi(n-1)}{\bar \Pi(n)}&=\te{\alpha},         &&\mbox{if $\Pi$ is lattice with span $1$,}
\end{align*}
and
$$
\lim_{u\to\infty}\frac{\bar \Pi^{*2}(u)}{\bar \Pi(u)}=2\int_0^\infty \te{\alpha y} \Pi(\td y),
$$
where $*$ denotes the convolution operator and $\bar \Pi(u)=\Pi((u,\infty))$. When $\alpha=0$, then we are in the subclass of subexponential measures and there is no need to distinguish between the lattice and non-lattice cases (see \cite{bertoin96}). We recall now \cite[Lemma 5.4]{palmowski11}.

\begin{lem}
\label{taillemma2}
Assume that for $i=0,1$ we have $\Pi_i\in \mathcal{S}^{(\alpha_i)}$ and $\psi_i(\alpha_i)<q$ for $\psi_i(\alpha_i) =\log E \te{\alpha_i X_i(1)}$.
Then
\begin{align*}
P(X_i(e_q)>x) &\sim \frac{q}{(q-\psi_i(\alpha_i))^2}\bar \Pi_i(x),\\
P(\overline X_i(e_q)>x) &\sim \frac{q}{(q-\psi_i(\alpha_i))^2}\frac{\Phi_i(q)+\alpha_i}{\Phi_i(q)}\bar \Pi_i(x),
\end{align*}
where $f(x)\sim g(x)$ means that $\lim_{x\to\infty}f(x)/g(x)=1$.
\end{lem}

If $G\in \mathcal{S}^{(\alpha)}$ then $\bar G(u)=\te{\alpha u}L(u)$ for a slowly varying function $L$. Hence, for $G_i\in \mathcal{S}^{(\alpha_i)}$ ($i=0,1$) with $\alpha_0<\alpha_1$
we have $\overline{G}_1(u)={\rm o}(\overline{G}_0(u))$.
Similarly, $\overline{G}_0(u)={\rm o}(\overline{G}_1(u))$ for $\alpha_1<\alpha_0$.
Moreover, it is known \cite{embrechts79} that if for independent  random variables $\chi_l$ ($l=1,2$) we have
$P(\chi_l>u)\sim c_l\bar G(u)$ as $u\to\infty$ and $G\in \mathcal{S}^{(\alpha)}$, then
$P(\chi_1+\chi_2>u)\sim (c_1E\te{\alpha \chi_2}+c_2E\te{\alpha \chi_1})\bar G(u)$. Let $\alpha =\max\{\alpha_0,\alpha_1\}$ and combine the above observations with  \eqref{bound1} in Lemma~\ref{tailbounds1} and Lemma~\ref{taillemma2} to obtain the following main result.

\begin{theorem}
Assume that $\Pi_i\in \mathcal{S}^{(\alpha_i)}$ and $\psi_i(\alpha)<q$ for $i=0,1$.
Moreover, let $F^{(2)}(y,\overline{j}) \leqslant F_0 (\geqslant 0)$ for any $y$ and $\overline{j}$,
and assume that there exist constants $c_i\geqslant 0$ such that $P(F^{(1)}(y, \bar j)>x) \sim P(F_0>x)\sim c_i \bar \Pi_i(x)$ as $x\rightarrow\infty$ for each $y$ and $\bar j$ (If $c_i=0$ then
$P(F^{(1)}(y, \bar j)>x) = \mathrm{o}(\bar \Pi_i(x))$). Then
\[
P(U>x) \sim \sum_{i=0,1} D_{i} \bar \Pi_i(x),
\]
as $x\rightarrow \infty$, where
\begin{align*}
D_i&=c_i\frac{q}{q-\psi_i(\alpha_i)}P(F^{(2)}(U, \overline{J})=i)+\frac{q}{(q-\psi_i(\alpha_i))^2}E\te{\alpha_i F^{(1)}(U, \bar J)}P(F^{(2)}(U, \overline{J})=i)\\
   &+\frac{q}{(q-\psi_i(\alpha_i))^2}\frac{\Phi_i(q)+\alpha_i}{\Phi_i(q)}P(F^{(2)}(U, \overline{J})=i) E\left(1-\te{-\alpha_i(-\underline{X}_i(e_q)-F^{(1)}(U, \bar J))}\right)\\
   &\quad P\left(-\underline{X}_i(e_q)-F^{(1)}(U, \bar J)>0, \quad F^{(2)}(U, \overline{J})=i\right).
\end{align*}
\end{theorem}

The conditions in this theorem are satisfied by both examples $F^{(2)}(y, \bar j)=0$ (in which case, we take $F_0=0,c_i=0$) and $F^{(2)}(y,\overline{j})=(B_i-y)^+$ (where $F_0=B_i$). If $\Pi_i$ is subexponential for at least one $\alpha_i$ (i.e.\ $\Pi_i\in \mathcal{S}^{(0)}$), then
$$
P(U>x)\sim \sum_{i=0,1}\left(c_i+\frac{1}{q}\right)\bar \Pi_i (x).
$$

We consider now the Cram\'{e}r case (light-tailed case). Assume that for each $i=0,1$ there exists $\Phi_i(q)$ such that
\begin{equation}\label{cramer1}
\psi_i(\Phi_i(q))=q
\end{equation}
and that
\begin{equation}\label{cramer2}
m_i(q) := \left.\frac{\partial \kappa_i(q,\beta)}{\partial \beta}\right|_{\beta=-\Phi_i(q)}<\infty,
\end{equation}
where $\kappa (\varrho,\zeta)$ is the Laplace exponent of a ladder height process
$\{(L^{-1}_i(t), H_i(t)), t\geqslant 0\}$ of $X_i$, that is:
$$
E\te{\varrho L^{-1}_i(t)+\zeta H_i(t)}=\te{\kappa_i (\varrho,\zeta)t}.
$$
Note that if $\Pi_i \in \mathcal{S}(\alpha)$ and $\psi_i(\alpha)<q$, then condition \eqref{cramer1} is not satisfied.
Moreover, we assume that
\begin{equation}\label{cramer3}
E\te{\Phi_i (q)F^{(1)}(U, \bar J)}<\infty, \qquad i=0,1.
\end{equation}

\begin{theorem}
Assume that \eqref{cramer1}-\eqref{cramer3} hold and that the supports of $\Pi_i$ ($i=0,1$) are non-lattice. Then, as $x\rightarrow \infty$, either
$$
P(U>x)\sim C \te{-\Phi(q) x}
$$
if $\Phi_0(q)\neq \Phi_1(q)$, where $\Phi(q):=\min\{\Phi_0(q), \Phi_1(q)\}$,
$$
C=\begin{cases}
C_0 &\mbox{for $\Phi_0(q)< \Phi_1(q)$,}\\
C_1 &\mbox{for $\Phi_0(q)> \Phi_1(q)$,}
\end{cases}
$$
for
$$
C_i=
P\left(-\underline{X}_i(e_q)>F^{(1)}(U, \overline{J})\quad\mbox{and}\quad  F^{(2)}(U, \overline{J})=i\right)\kappa(q,0)\left(\Phi_i(q)
m_i(q)
\right)^{-1},\qquad i=0,1,
$$
and
$$
P(U>x)\sim (C_0+C_1) \te{-\Phi(q) x},
$$
otherwise.
\end{theorem}
\begin{proof}
From the proof of \cite[Th.\ 5.2]{palmowski11} it follows that
$$
\lim_{x\to\infty}\te{\Phi_i(q)x}P(\overline X_i(e_q)>x)=\frac{\kappa_i(q,0)}{\Phi(q)m_i(q)}.
$$
Note that by \eqref{cramer1} and \eqref{cramer3}, we have $P(X_i(e_q)+F^{(1)}(U, \bar J)>x)={\rm o}(\te{-\Phi_i(q)x})$. Inequalities \eqref{bound3} and \eqref{bound4} in Lemma~\ref{tailbounds1} complete the proof.
\end{proof}

\phantomsection
\addcontentsline{toc}{section}{Acknowledgments}
\subsection*{Acknowledgments}
The work of Zbigniew Palmowski is partially supported by the Ministry of Science and Higher Education of Poland under the grant N N201 394137 (2009-2011). 
The work of Bert Zwart is supported by an NWO VIDI grant and an IBM faculty award.
Bert Zwart is also affiliated with \textsc{Eurandom}, VU University Amsterdam, and Georgia Institute of Technology.

\phantomsection
\addcontentsline{toc}{section}{References}

\begin{thebibliography}{10}

\bibitem{altman99}
{\sc Altman, E., Avrachenkov, K., Barakat, C. and Núñez-Queija, R.} (1999).
\newblock State-dependent m/g/1 type queueing analysis for congestion control
  in data networks.
\newblock {\em RFC\/} {\bf 2581,}.

\bibitem{asmussen-APQ}
{\sc Asmussen, S.} (2003).
\newblock {\em Applied Probability and Queues}.
\newblock Springer-Verlag, New York.

\bibitem{bekker08}
{\sc Bekker, R., Boxma, O.~J. and Kella, O.} (2008).
\newblock Queues with delays in two-stage startegies and {L}\'{e}vy input.
\newblock {\em Journal of Applied Probability\/} {\bf 45,} 314--332.

\bibitem{bekker09}
{\sc Bekker, R., Boxma, O.~J. and Resing, J. A.~C.} (2009).
\newblock L\'evy processes with adaptable exponent.
\newblock {\em Advances in Applied Probability\/} {\bf 41,} 177--205.

\bibitem{bertoin-LP}
{\sc Bertoin, J.} (1996).
\newblock {\em {L}\'{e}vy Processes}.
\newblock No.~121 in Cambridge tracts in mathematics. Cambridge University
  Press.

\bibitem{bertoin96}
{\sc Bertoin, J. and Doney, R.~A.} (1996).
\newblock Some asymptotic results for transient random walks.
\newblock {\em Advances in Applied Probability\/} {\bf 28,} 207--226.

\bibitem{chaumont09}
{\sc Chaumont, L., Kyprianou, A.~E. and Pardo, J.~C.} (2009).
\newblock Some explicit identities associated with positive self-similar
  {M}arkov processes.
\newblock {\em Stochastic Processes and Their Applications\/} {\bf 119,}
  980--1000.

\bibitem{debicki12}
{\sc D\c{e}bicki, K., M.~M.} (2012).
\newblock L\'evy driven queues.
\newblock In {\em Surveys in Operations Research and Management Science}. ed.
  J.~K. Lenstra, M.~Trick, and B.~Zwart.
\newblock To appear.

\bibitem{dshalalow97}
{\sc Dshalalow, J.~H.} (1997).
\newblock Queueing systems with state dependent parameters.
\newblock In {\em Frontiers in queueing}.
\newblock Probab. Stochastics Ser. CRC, Boca Raton, FL pp.~61--116.

\bibitem{embrechts79}
{\sc Embrechts, P., Goldie, C.~M. and Veraverbeke, N.} (1979).
\newblock Subexponentiality and infinite divisibility.
\newblock {\em Zeitschrift f\"ur Wahrscheinlichkeitstheorie und Verwandte
  Gebiete\/} {\bf 49,} 335--347.

\bibitem{foss04}
{\sc Foss, S. and Konstantopoulos, T.} (2004).
\newblock An overview of some stochastic stability methods.
\newblock {\em Journal of the Operations Research Society of Japan\/} {\bf 47,}
  275--303.

\bibitem{guillemin04}
{\sc Guillemin, F., Robert, P. and Zwart, B.} (2004).
\newblock A{IMD} algorithms and exponential functionals.
\newblock {\em The Annals of Applied Probability\/} {\bf 14,} 90--117.

\bibitem{hubalek11}
{\sc Hubalek, F. and Kyprianou, A.~E.} (2011).
\newblock Old and new examples of scale functions for spectrally negative
  l{\'e}vy processes.
\newblock In {\em Seminar on Stochastic Analysis, Random Fields and
  Applications VI}. ed. R.~Dalang, M.~Dozzi, and F.~Russo.
\newblock Progress in Probability. Springer pp.~119--145.

\bibitem{kella03}
{\sc Kella, O., Perry, D. and Stadje, W.} (2003).
\newblock A stochastic clearing model with a {B}rownian and a compound
  {P}oisson component.
\newblock {\em Probability in the Engineering and Informational Sciences\/}
  {\bf 17,} 1--22.

\bibitem{kyprianou-ILFLP}
{\sc Kyprianou, A.~E.} (2006).
\newblock {\em Introductory Lectures on Fluctuations of {L}\'evy Processes with
  Applications}.
\newblock Universitext. Springer-Verlag, Berlin.

\bibitem{kyprianou05}
{\sc Kyprianou, A.~E. and Palmowski, Z.} (2005).
\newblock A martingale review of some fluctuation theory for spectrally
  negative {L}\'evy processes.
\newblock In {\em S\'eminaire de {P}robabilit\'es {XXXVIII}}.
\newblock vol.~1857 of {\em Lecture Notes in Math}. Springer, Berlin
  pp.~16--29.

\bibitem{kyprianou08a}
{\sc Kyprianou, A.~E. and Rivero, V.} (2008).
\newblock Special, conjugate and complete scale functions for spectrally
  negative l\'evy processes.
\newblock {\em The Electronic Journal of Probability\/}.

\bibitem{lambert00}
{\sc Lambert, A.} (2000).
\newblock Completely asymmetric {L}\'evy processes confined in a finite
  interval.
\newblock {\em Annales de l'Institut Henri Poincar\'e. Probabilit\'es et
  Statistiques\/} {\bf 36,} 251--274.

\bibitem{maulik06}
{\sc Maulik, K. and Zwart, B.} (2006).
\newblock Tail asymptotics for exponential functionals of {L}\'evy processes.
\newblock {\em Stochastic Processes and their Applications\/} {\bf 116,}
  156--177.

\bibitem{maulik09}
{\sc Maulik, K. and Zwart, B.} (2009).
\newblock An extension of the square root law of {TCP}.
\newblock {\em Annals of Operations Research\/} {\bf 170,} 217--232.

\bibitem{palmowski11}
{\sc Palmowski, Z. and Vlasiou, M.} (2011).
\newblock A {L}\'{e}vy input model with additional state-dependent services.
\newblock {\em Stochastic Processes and their Applications\/} {\bf 121,}
  1546--1564.

\bibitem{pistorius}
{\sc Pistorius, M.} (2003).
\newblock Exit problems of {L}\'evy processes with applications in finance.
\newblock {\em PhD thesis}.
\newblock Utrecht University The Netherlands.

\bibitem{suprun76}
{\sc Suprun, V.~N.} (1976).
\newblock The ruin problem and the resolvent of a killed independent increment
  process.
\newblock {\em Ukrain. Mat. \v Z.\/} {\bf 28,} 53--61, 142.

\bibitem{vlasiou10}
{\sc Vlasiou, M. and Palmowski, Z.} (2010).
\newblock Tail asymptotics for a random sign {L}indley recursion.
\newblock {\em Journal of Applied Probability\/} {\bf 47,} 72--83.

\end{thebibliography}

\end{document}